\documentclass[12pt]{amsart}
\usepackage[utf8]{inputenc}
\usepackage[lite, alphabetic, nobysame]{amsrefs}
\usepackage{amsmath,mathtools}
\usepackage{amssymb}
\usepackage{srcltx}
\usepackage{xcolor}
\usepackage{xspace}
\usepackage{enumerate}
\usepackage{geometry}
\usepackage{marginnote}

\usepackage{float}
\usepackage{tikz}
\usepackage{tikz-cd}
\usepackage{wrapfig}
\usetikzlibrary{positioning}

\usepackage{amsmath, amsthm}
\usepackage{amssymb}
\usepackage{todonotes}

\usepackage{cite}

\usepackage{tikz}
\usetikzlibrary{positioning}

\theoremstyle{definition}

\newtheorem{thm}{Theorem}[section]
\newtheorem*{thm*}{Theorem}
\newtheorem{lemma}[thm]{Lemma}

\newtheorem{defn}[thm]{Definition}

\newtheorem{claim}[thm]{Claim}

\newtheorem{remark}[thm]{Remark}
\newtheorem{fact}[thm]{Fact}

\renewcommand{\subset}{\subseteq}

\newcommand\N{\mathbb{N}}

\newcommand{\set}[2]{ \left\{ #1 :\, #2 \right\} }
\newcommand{\seqq}[2]{ \left( #1 :\, #2\right) }
\newcommand{\seq}[1]{ \left( #1 \right) }

\title{Generic dichotomy for homomorphisms for $E_0^\mathbb{N}$}

\author{Assaf Shani}
\address{Department of Mathematics and Statistics, Concordia University University, Montreal, QC  H3G 1M8, Canada}
\email{assaf.shani@concordia.ca}
\urladdr{https://sites.google.com/view/assaf-shani/}

\date{\today}



\thanks{Research partially supported by NSF grant DMS-2246746 and NSERC grant RGPIN-2024-05827.}

\begin{document}

\begin{abstract}
We prove the following dichotomy. Given an analytic equivalence relation $E$, either ${E_0^{\mathbb{N}}}\leq_B{E}$ or else any Borel homomorphism from $E_0^{\mathbb{N}}$ to $E$ is ``very far from a reduction'', specifically, it factors, on a comeager set, through the projection map $(2^{\mathbb{N}})^{\mathbb{N}}\to (2^{\mathbb{N}})^k$ for some $k\in\mathbb{N}$. As a corollary, we prove that $E_0^{\mathbb{N}}$ is a prime equivalence relation, answering a question on Clemens.
\end{abstract}
\maketitle

\section{Introduction}~\label{section: Introduction}
Let $E$ and $F$ be equivalence relations on Polish spaces $X$ and $Y$ respectively. A map $f\colon X\to Y$ is said to be a \textbf{reduction} of $E$ to $F$ if for any $x_1,x_2\in X$, 
\begin{equation*}
    x_1\mathrel{E}x_2\iff f(x_1)\mathrel{F}f(x_2).
\end{equation*}
We say that $E$ is \textbf{Borel reducible} to $F$, denoted ${E}\leq_B{F}$ if there is a Borel measurable function which is a reduction of $E$ to $F$. Borel reducibility is the most central concept in the study of equivalence relations on Polish spaces.

A map $f\colon X\to Y$ is a \textbf{homomorphism} from $E$ to $F$, if for any $x_1,x_2\in X$,
\begin{equation*}
x_1\mathrel{E}x_2 \implies f(x_1)\mathrel{F} f(x_2).    
\end{equation*}
We write $f\colon E\to_B F$ to denote that $f$ is a Borel measurable homomorphism from $E$ to $F$. To prove an irreducibility result, say that some $E$ is not Borel reducible to $F$, many times the argument takes the following outline: take an arbitrary Borel homomorphism from $E$ to $F$, and prove that it cannot be a reduction.

This paper concerns the equivalence relation $E_0^\N$, often denoted by $E_3$, which plays a central role in the theory of Borel equivalence relations. For more background the reader is referred to \cite{Hjorth-Kechris-New-Dichotomies, Hjorth-Kechris-recent-developments-2001, Kano08, Gao09}. For instance, the dichotomy proved in \cite{Hjorth-Kechris-recent-developments-2001} implies that $E_0^\N$ is an immediate successor of $E_0$ with respect to $\leq_B$.

Recall that $E_0$ on $2^\N$ is defined as the eventual equality relation between binary sequences. Equivalently, $E_0$ is the orbit equivalence relation induced by the action  $\bigoplus_{n\in\N}\mathbb{Z}_2 \curvearrowright 2^\N$. For a set $X$, define $E_0^X$ on $(2^\N)^X$ as the product equivalence relation. Equivalently, $E_0^X$ is the orbit equivalence relation induced by the point-wise action $(\bigoplus_{n\in\N}\mathbb{Z}_2)^X\curvearrowright (2^\N)^X$, where $(\bigoplus_{n\in\N}\mathbb{Z}_2)^X$ is the (full support) product of $X$ many copies of $\bigoplus_{n\in\N}\mathbb{Z}_2$.

For $k\in\N$, let $\pi_k\colon (2^\N)^\N \to (2^\N)^k$ be the projection map. Note that $\pi_k \colon E_0^\N \to_B E_0^k$ is Borel homomorphism. For each $k\in \N$, $E_0^{k}$ is Borel bireducible with $E_0$. So $\set{\pi_k}{k\in\N}$ is a family of Borel homomorphisms from $E_0^\N$ to a Borel equivalence relation which does not reduce $E_0^\N$. We prove that, generically, these are essentially all the Borel homomorphisms from $E_0^\N$ to analytic equivalence relations which do not reduce $E_0^\N$.

\begin{thm}\label{thm: main}
    Let $E$ be an analytic equivalence relation. Either
    \begin{itemize}
        \item $E^\N_0$ is Borel reducible to $E$, or
        \item for any Borel homomorphism $f\colon E_0^\N \to_B E$ there is $k\in\N$ so that $f$ factors through $\pi_k$ on a comeager set, that is, there is a Borel homomorphism $h\colon E_0^k \to_B E$, defined on a comeager set, so that for comeager many $x\in (2^\N)^\N$,
        \begin{equation*}
        h\circ \pi_k(x) \mathrel{E} f(x).
        \end{equation*}
    \end{itemize}
\end{thm}
\begin{figure}[H]\label{figure: main thm factoring}
    \centering
\begin{tikzpicture}[
node/.style={}]
\node[node]      (E3)   at (0,0)        {$E_0^\N$};
\node[node]     (E0k) [below=of E3]
{$E_0^k$};
\draw[->]  (E3.south) --node[anchor=east] {$\pi_k$} (E0k.north) ;
\node[node]     (fplus) [right=of E0k]
{$E$};
\draw[->] (E3.south east) -- node[anchor=south west] {$f$} (fplus.north west)   ;
\draw[->,dashed]    (E0k.east) --node[anchor=south] {$h$} (fplus.west) ;
\end{tikzpicture}

    \caption{$(\forall f\colon E_0^\N\to_B E )(\exists k\in\N\, \exists h\colon E_0^k\to_B E)$} 
    %\label{fig:enter-label}
\end{figure}

\subsection{Primeness for equivalence relations}

From the point of view of the study of Borel reducibility as the study of definable cardinality of quotients of Polish spaces, a Borel homomorphism corresponds to a definable map between two such quotients, and a Borel reduction corresponds to an injective map.

\begin{defn}[Clemens~\cite{Clemens-primeness-2020}] Let $E$ and $F$ be Borel equivalence relations on Polish spaces $X$ and $Y$ respectively. Say that {\bf $E$ is prime to $F$} if for any Borel homomorphism $f\colon E\to_B F$, $E$ retains its complexity on a fiber, that is, there is $y\in Y$ so that $E$ is Borel reducible to $E\restriction \set{x\in X}{f(x)\mathrel{F}y}$.
\end{defn}
Primeness is a strong form of Borel-irreducibility, which holds between many pairs of benchmark equivalence relations (see~\cite[Theorem~1]{Clemens-primeness-2020}).

In the classical context of cardinality, primeness corresponds to a pigeonhole principle: any function $f\colon A\to B$ has a fiber of cardinality $|A|$. This is true if and only if the cardinality $|B|$ is strictly smaller than the cofinality of $|A|$. Recall that the cardinality $|A|$ is regular if it is equal to its cofinality. This is true if and only if for any $|B|<|A|$, any function from $A$ to $B$ has a fiber of size $|A|$.

Following this analogy Clemens defined \textbf{regular} equivalence relation as follows. In the context of definable cardinality, when not every two sizes are comparable, the stronger notion of a \textbf{prime} equivalence relation is also of interest.

\begin{defn}[Clemens~\cite{Clemens-primeness-2020}]\label{Definition : prime} 
Let $E$ be a Borel equivalence relation.
\begin{itemize}
    \item $E$ is {\bf prime} if for any Borel equivalence relation $F$, either $E\leq_B F$ or $E$ is prime to $F$.
    \item $E$ is {\bf regular} if for any Borel equivalence relation $F$, if $F<_BE$ then $E$ is prime to $F$.
\end{itemize}
\end{defn}
For example, it follows from the celebrated $E_0$-dichotomy~\cite{HKL90} that $E_0$ is prime. The $E_0^\N$ dichotomy proved by Hjorth and Kechris~\cite{Hjorth-Kechris-recent-developments-2001} implies that $E_0^\N$ is regular. (See \cite[Section~4]{Clemens-primeness-2020}.) 
Clemens~\cite[Question~4.2]{Clemens-primeness-2020} asked if $E_0^\N$ is in fact prime. 
\begin{thm}\label{thm: E3 prime}
    For any analytic equivalence relation $E$, either ${E_0^\N}\leq_B{E}$ or $E_0^\N$ is prime to $E$. In particular, $E_0^\N$ is prime.
\end{thm}
\begin{proof}
Fix an analytic equivalence relation $E$ which does not Borel reduce $E_0^\N$, and fix a Borel homomorphism $f\colon E_0^\N\to E$. By Theorem~\ref{thm: main}, there is $k\in\N$ and a comeager set $C\subset (2^\N)^\N$ so that $f$ factors through $\pi_k$ on $C$. We identify $(2^\N)^\N$ with $(2^\N)^k \times (2^\N)^{\N\setminus k}$. By the Kuratowski-Ulam theorem~\cite[Theorem 8.41 (iii)]{Kechris-DST-1995} there is $y\in (2^\N)^k$ so that $C_y = \set{z\in (2^\N)^{\N\setminus k}}{(y,z)\in C}$ is comeager in $(2^\N)^{\N\setminus k}$. Note that $\{y\}\times C_y$ is contained in a fiber of $f$. We conclude the proof by showing that ${E_0^\N}\leq_B E_0^\N\restriction \{y\}\times C_y$.

Since $E_0^\N \restriction \{y\}\times (2^\N)^{\N\setminus k}$ is isomorphic, via a homeomorphism, to $E_0^\N$, it suffices to show that ${E_0^\N} \leq_B {E_0^\N\restriction C}$ for any comeager set $C$ in the domain of $E_0^\N$. We will give a proof of this fact in Section~\ref{subsec: E3 retains complexity} below. Here we mention that it follows from the $E_0^\N$ dichotomy of Hjorth and Kechris~\cite{Hjorth-Kechris-recent-developments-2001}, as $E_0^\N \restriction C$ is not Borel reducible to $E_0$, for any comeager set $C$. 
\end{proof}

\section{A generic dichotomy for homomorphisms for $E_0$}
In this section we note that the primeness dichotomy of $E_0$ is also true for all analytic equivalence relations. This follows from the following generic dichotomy for Borel homomorphisms.
\begin{thm}\label{thm: homom dichotomy E0}
Let $E$ be an analytic equivalence relation. Then either
    \begin{itemize}
        \item ${E_0}\leq_B{E}$ or 
        \item any Borel homomorphism $f\colon E_0 \to_B E$ sends a comeager subset of $2^\N$ into a single $E$-class.
    \end{itemize}
\end{thm}
\begin{remark}
The theorem follows from the Ulm-invariants dichotomy of Hjorth and Kechrish~\cite{Hjorth-Kechris-Ulm-1995}, since if $E$ is Ulm classifiable then the second bullet holds. %\todo{reference for this?}
The dichotomy in \cite{Hjorth-Kechris-Ulm-1995} is proved assuming the existence of sharps for reals.

We include a direct proof below (not using any set theoretic assumptions). The dichotomy for homomorphisms is much easier to prove than the $E_0$-dichotomies. In fact, these ideas can be found as part of the proof of any $E_0$-dichotomy. 

This is one of the motivations to study such generic dichotomies for homomorphisms. While, beyond $E_0$, there are no more dichotomies quite like the $E_0$-dichotomy (see \cite[Theorem~5.1]{Kechris-Louveau-1997}), a generic analysis of all Borel homomorphisms is one aspect of the $E_0$-dichotomy which we can be generalized beyond $E_0$.
\end{remark}

Given two equivalence relations $F$ and $E$ on the same domain, say that $E$ \textbf{extends} $F$ if $F\subset E$. Let $F$ and $E$ be equivalence relations on domains $X$ and $Y$ respectively. Given a function $f\colon X\to Y$, define the pullback of $E$ as the equivalence relation $E^\ast$ on $X$ defined by $x\mathrel{E^\ast} y\iff f(x) \mathrel{E} f(y)$. Note that $f$ is a Borel homomorphism from $E$ to $F$ if and only if $E^\ast$ extends $E$. Furthermore, $f$ sends a comeager subset of $X$ to a single $E$ class if and only if $E^\ast$ has a comeager class. Theorem~\ref{thm: homom dichotomy E0} is equivalent to the following.

\begin{thm}\label{thm: E0 dichotomy of homoms}
    Let $E$ be an analytic equivalence relation on $2^\N$ which extends $E_0$. Then either
    \begin{itemize}
        \item ${E_0}\leq_B{E}$ or 
        \item $E$ has a comeager class.
    \end{itemize}
\end{thm}

\begin{fact}\label{fact: E0 respecting splitting}
    Let $C\subset 2^\N \times 2^\N$ be comeager. There is a Borel homomorphism $f\colon E_0 \to_B E_0$  so that if $x\not\mathrel{E_0}y$ then $(f(x),f(y))\in C$.
\end{fact}
This fact is commonly used when building reductions from $E_0$. For example, the map $\alpha_1$ defined in Section~\ref{subsec: main construction}, using the set $D_{0,1} = C \subset (2^\N)^2$ in that construction, satisfies the conclusion in Fact~\ref{fact: E0 respecting splitting}.

To prove Theorem~\ref{thm: E0 dichotomy of homoms}, let $E$ be an analytic equivalence relation which extends $E_0$ and does not have a comeager class. It follows that $C = 2^\N\times 2^\N \setminus E$ is comeager. Then $f$ as above is a reduction of $E_0$ to $E$.

\section{A generic dichotomy for homomorphisms for $E_0^\N$}
Towards the proof of Theorem~\ref{thm: main}, we begin with some technical lemmas.
\subsection{Symmetries of $E_0^\N$}
In this section we prove a lemma regarding Vaught transforms for the action $(\bigoplus_{n\in\N}\mathbb{Z}_2)^\N \curvearrowright (2^\N)^\N$.

Let $X$ and $Y$ be infinite sets and consider the space $(2^X)^Y$ with the product topology. Let $G = (\bigoplus_{x\in X}\mathbb{Z}_2)^Y$, acting on $(2^X)^Y$ in the natural way. We consider $G$ as a topological group with the product topology, where $\bigoplus_{x\in X}\mathbb{Z}_2$ is taken with the discrete topology. Let $\Gamma$ be the subgroup of $G$ of all finite support sequences. That is, $g\in\Gamma$ if $g(y)$ is the identity for all but finitely many $y\in Y$.

Given a subset $Y_0\subset Y$, identify $(2^X)^Y$ with $(2^X)^{Y_0}\times(2^X)^{Y\setminus Y_0}$. For $a\in (2^X)^{Y_0}$ and $D\subset (2^X)^Y$, define $D_{a} = \set{b\in (2^X)^{Y\setminus Y_0}}{(a,b)\in D}\subset (2^X)^{Y\setminus Y_0}$.
\begin{lemma}\label{lemma: permutations}
    Fix a dense open $D\subset (2^X)^Y$ and $\zeta\in (2^X)^Y$. Assume that for any finite $Y_0\subset Y$ and for any $\gamma\in\Gamma$, 
    \begin{equation*}
        D_{\gamma\cdot \zeta\restriction Y_0} \subset (2^X)^{Y\setminus Y_0}
    \end{equation*}
    is not empty. Then 
    \begin{equation*}
        \set{g\in G}{g\cdot \zeta\in D}
    \end{equation*}
    is dense open in $G$.  

    In particular, if $D$ is assumed to be comeager, then we conclude that $\set{g\in G}{g\cdot \zeta\in D}$ is comeager.
\end{lemma}

\begin{proof}
First, as the map $G\to (2^X)^Y$, $g \mapsto g\cdot\zeta$, is continuous, then $\set{g\in G}{g\cdot \zeta\in D}$ is open as the pre-image of $D$. To show that $\set{g\in G}{g\cdot \zeta\in D}$ is dense, fix a finite set $Y_0\subset Y$ and some $\pi \in (\bigoplus_{x\in X}\mathbb{Z}_2)^{Y_0}$. We need to find some $g\in G$ extending $\pi$ so that $g\cdot \zeta \in D$.

By assumption, $D_{\pi\cdot\zeta \restriction Y_0}\subset (2^X)^{Y\setminus Y_0}$ is a non-empty open set. As all orbits of the action $(\bigoplus_{x\in X}\mathbb{Z}_2)^{Y\setminus Y_0}\curvearrowright (2^X)^{Y\setminus Y_0}$ are dense, we may find some $g\in G$ extending $\pi$ so that $(g\restriction Y\setminus Y_0)\cdot (\zeta \restriction Y\setminus Y_0) \in D_{\pi\cdot\zeta \restriction Y_0}$, and therefore $g\cdot \zeta \in D$.
\end{proof}

\subsection{A reformulation}
For $k\in\N$, we can view $E_0^k$ as an equivalence relation on $(2^\N)^\N$, defined by $x\mathrel{E_0^k}y\iff x\restriction k \mathrel{E_0^k}y\restriction k$. That is, by identifying $E_0^k$ with its pullback via the homomorphism $\pi_k\colon E_0^\N \to_B E_0^k$. We therefore view 
\begin{equation*}
    E_0 \supseteq E_0^2 \supseteq E_0^3 \supseteq \dots \supseteq E_0^\N
\end{equation*}
as a descending sequence of equivalence relations on $(2^\N)^\N$.

We will prove Theorem~\ref{thm: main} in the following equivalent form.
\begin{thm}\label{thm: main reformulated}
Let $E$ be an analytic equivalence relation on $(2^\N)^\N$ which extends $E_0^\N$. Then either
\begin{itemize}
    \item ${E_0^\N}\mathrel{\leq_B}{E}$ or
    \item there is $k\in\N$ so that $E$ extends $E_0^k$ on a comeager set.
\end{itemize}
\end{thm}
\begin{proof}[Proof of Theorem~\ref{thm: main} from Theorem~\ref{thm: main reformulated}]
Let $E$ be an analytic equivalence relation and $f\colon E_0^\N\to_B E$ a Borel homomorphism, and assume that $E_0^\N$ is not Borel reducible to $E$. Let $E^\ast$ on $(2^\N)^\N$ be the pullback of $E$. Then $E^\ast$ extends $E_0^\N$, and $E_0^\N$ is not Borel reducible to  $E^\ast$. By Theorem~\ref{thm: main reformulated} there is some $k\in\N$ and a comeager $C\subset (2^\N)^\N$ so that for $x,y\in C$,
\begin{equation*}
    x\restriction k \mathrel{E_0^k} y\restriction k \implies x\mathrel{E^\ast}y.
\end{equation*}
Let $C_k$ be the set of all $\xi\in(2^\N)^k$ so that the fiber $C_\xi\subset (2^\N)^{\N\setminus k}$ is comeager. Fix a Borel map $g\colon C_k \to C$ so that $g(\xi)\restriction k = \xi$. Then $g$ is a homomorphism from $E_0^k$ to $E^\ast$, and therefore $h = f\circ g$ is a homomorphism from $E_0^k$ to $E$, defined on a comeager set. Now for any $x\in C$, $h\circ \pi_k(x) \mathrel{E} f(x)$, as required.
\end{proof}

Towards the proof of Theorem~\ref{thm: main reformulated}, fix an analytic equivalence relation $E$ on $(2^\N)^\N$ which extends $E_0^\N$. Assume that for every $k\in\N$, $E$ does not extend $E_0^k$ on a comeager set. We need to prove that $E_0^\N$ is Borel reducible to $E$.

\begin{lemma}\label{lemma: generics Einequiv}
For every $k$, there is a comeager set $C_k\subset (2^\N)^k \times (2^\N)^{\N\setminus k} \times (2^\N)^{\N\setminus k}$ so that $(x,y)\not\mathrel{E}(x,z)$ for any $(x,y,z)\in C_k$.
\end{lemma}
\begin{proof}
    Otherwise, since the actions $(\bigoplus_{n\in\N}\mathbb{Z}_2)^k\curvearrowright (2^\N)^k$ and $(\bigoplus_{n\in\N}\mathbb{Z}_2)^{\N\setminus k}\curvearrowright (2^\N)^{\N\setminus k}$ have dense orbits, we would get a comeager set $C\subset (2^\N)^k \times (2^\N)^{\N\setminus k} \times (2^\N)^{\N\setminus k}$, which we may assume is invariant, so that $(x,y)\mathrel{E}(x,z)$ for all $(x,y,z)\in C$. Now $E$ extends $E_0^k$ on the comeager set of all $(x,y)\in (2^\N)^k \times (2^\N)^{\N\setminus k}$ for which $\set{z\in (2^\N)^{\N\setminus k}}{(x,y,z)\in C}$ is comeager, contradicting our assumption.
\end{proof}
We identify each $C_k$ as a subset of
\begin{equation*}
C_k\subset (2^\N)^k \times (2^\N)^{2\times (\N\setminus k)}.    
\end{equation*}
For each $k\in\N$, $m\geq k$, and $\gamma\in (\bigoplus_{n\in\N}\mathbb{Z}_2)^{k}\times (\bigoplus_{n\in\N}\mathbb{Z}_2)^{2\times (m\setminus k)}$, consider the set
\begin{equation*}
\begin{split}
    \{ & (\eta_0,\dots,\eta_{k-1},\xi_k,\zeta_k,\dots,\xi_{m-1},\zeta_{m-1}) \in (2^\N)^k \times (2^\N)^{2\times (m\setminus k)} \colon  \\ & (C_k)_{\gamma\cdot (\eta_0,\dots,\eta_{k-1},\xi_k,\zeta_k,\dots,\xi_{m-1},\zeta_{m-1})}\subset (2^\N)^{2\times \N\setminus m}\textrm{ is comeager} \},
\end{split}    
\end{equation*}
a comeager subset of $(2^\N)^k \times (2^\N)^{2\times (m\setminus k)}$. Let 
\begin{equation*}
D_{k,m}\subset (2^\N)^k \times (2^\N)^{2\times (m\setminus k)}    
\end{equation*}
be the intersection of all (countably many) such sets. Write
\begin{equation*}
    D_{k,m} = \bigcap_{l\in\N} D_{k,m}^l,
\end{equation*}
an intersection of dense open sets. We may assume that, for $k < m < l < h$,
\begin{equation*}
    D_{k,m}^{l} \times (2^\N)^{2\times (l\setminus m)} \supseteq D_{k,l}^l \supseteq D_{k,l}^{h}.
\end{equation*}
We will identify members of $(2^{<\N})^k \times (2^{<\N})^{2\times (m\setminus k)}$ with the basic open subsets of $(2^\N)^k \times (2^\N)^{2\times (m\setminus k)}$ which they define.

\subsection{A construction}\label{subsec: main construction}
We want to find a Borel homomorphism $f\colon E_0^\N \to_B E_0^\N$ which is a reduction from $E_0^\N$ to $E$. Roughly speaking, we hope to construct such $f$ so that, given $x\not\mathrel{E_0^\N}y$, then $f(x)$ and $f(y)$ can be written as $(a,b), (a,c) \in (2^\N)^k\times (2^\N)^{\N\setminus k}$ respectively, so that $(a,b,c)\in C_k$. Instead we will ensure that $(a,b,c)$ satisfies the assumptions in Lemma~\ref{lemma: permutations}, with respect to the comeager set $C_k$.

We will construct $f$ so that $f(x)(n)$ will depend on $x\restriction n+1$. We view $x\restriction n \in (2^\N)^{n}$ which we identify as $(2^n)^\N$. The equivalence relation $E_0^n$, identified on $(2^n)^\N$, is still ``eventual equality'', between sequences of members of $2^n$. The point here is that given $x,y\in (2^\N)^\N$, $k,l\in\N$, for which $x(k)(l)\neq y(k)(l)$, then $(x\restriction n)(k)(l) \neq (y\restriction n)(k)(l)$ for all $n\geq k+1$, where $(x\restriction n)(k)(l)\in 2^n$.

First, we define maps $\alpha_n\colon (2^n)^\N \to 2^\N$ as follows. We define recursively maps $\alpha_n \colon (2^n)^r \to 2^{<\N}$ which cohere, that is, for $r_1 < r_2$, $t_1\in (2^n)^{r_1}$, and $t_2\in (2^n)^{r_2}$ extending $t_1$, $\alpha_n(t_2)$ extends $\alpha_n(t_1)$. Then $\alpha_n\colon (2^n)^\N \to 2^\N$ will be defined as the limit.

At stage $r$, assume that we have defined
\begin{equation*}
    \alpha_n \colon (2^n)^r \to 2^{<\N},\, n\leq r.
\end{equation*}
For each $k<m\leq r$, since $D_{k,r+1}^{r+1} \subset (2^\N)^k \times (2^\N)^{2\times(r+1\setminus k)}$ is dense open, any member of $(2^{<\N})^k \times (2^{<\N})^{2\times (r+1\setminus k)}$ can be extended to define a subset of $D_{k,r+1}^{r+1}$. By extending finitely many times, we may find
\begin{equation*}
    a^n_{\xi}\in 2^{<\N}, \textrm{ for }\xi\in 2^n, n\leq r+1,
\end{equation*}
so that for any $k < r+1$, for any
\begin{equation*}
    \seqq{t_n \in (2^n)^r}{n<k}, \seqq{t_n,s_n \in (2^n)^r}{k\leq n \leq r}, \textrm{ and any}
\end{equation*}
\begin{equation*}
    \seqq{\xi_n\in 2^n}{n<k}, \seqq{\xi_n\neq \zeta_n \in 2^n}{k\leq n \leq r+1},
\end{equation*}
\begin{equation*}
    \left(\seqq{\alpha_n(t_n)\frown a^n_{\xi_n}}{n<k},\seqq{\alpha_n(t_n)\frown a^n_{\xi_n}, \alpha_n(s_n)\frown a^n_{\zeta_n}}{k\leq n \leq r},(a^{r+1}_{\xi_{r+1}},a^{r+1}_{\zeta_{r+1}})\right) \in D^{r+1}_{k,r+1}.
\end{equation*}
This concludes the definition of the maps $\alpha_n$, $n\in\N$. Note that for each $n\in\N$, $\alpha_n \colon E_0^n \to_B E_0$ is a Borel homomorphism.
\begin{remark}
    We may assume that $D_{0,n} \subset (2^\N)^{2\times n} \setminus E_0^n$, and so $\alpha_n$ is a reduction of $E_0^n$ to $E_0$.
\end{remark}
\begin{claim}\label{claim: main construction genericity}
    Suppose $x,y\in (2^\N)^\N$ are such that $x\restriction k = y \restriction k$ and $x(k)\not\mathrel{E_0}y(k)$. Then for any $k<m$,
    \begin{equation*}
        \left(\seqq{\alpha_n(x\restriction n+1)}{n<k}, \seqq{\alpha_n(x\restriction n+1),\alpha_n(y\restriction n+1)}{k\leq n <m}\right) \in D_{k,m}
    \end{equation*}
\end{claim}
\begin{proof}
    It suffices to prove membership in $D_{k,m}^r$ for infinitely many $r\in\N$, since $D^r_{k,m}$ is decreasing in $r$. For each $r$ so that $x(k)(r)\neq y(k)(r)$, we have that $(x\restriction n+1) (r) \neq (y\restriction n+1) (r)$, as members of $2^{n+1}$, for all $k\leq n$. Therefore at stage $r$ of the construction we ensure that 
    \begin{equation*}
        \left(\seqq{\alpha_n(x\restriction n+1)}{n<k}, \seqq{\alpha_n(x\restriction n+1),\alpha_n(y\restriction n+1)}{k\leq n <m}\right)\in D^r_{k,m},
    \end{equation*}
since $D_{k,m}^{r} \times (2^\N)^{2\times (r\setminus m)} \supseteq D_{k,r}^r$.
\end{proof}
Finally, define $f\colon (2^\N)^\N \to (2^\N)^\N$ by
\begin{equation*}
    f(x)(n) = \alpha_n(x\restriction n+1).
\end{equation*}
Then $f\colon E_0^\N \to_B E_0^\N$ is Borel homomorphism. To conclude the proof of the main theorem, we prove that $f$ is a reduction of $E_0^\N$ to $E$.
\subsection{Concluding the proof}
Since $E$ extends $E_0^\N$, it remains to prove that if $x\not\mathrel{E_0^\N}y$ then $f(x)\not\mathrel{E}f(y)$. Since $f\colon E_0^\N \to_B E$ is a homomorphism, it suffices to prove the following.
\begin{claim}
    Suppose $x,y\in (2^\N)^\N$, $x\restriction k = y\restriction k$ and $x(k)\not\mathrel{E_0}y(k)$. Then $f(x)\not\mathrel{E}f(y)$.
\end{claim}
\begin{proof}
By the definition of $f$, Claim~\ref{claim: main construction genericity}, and the choice of the sets $D_{k,m}$, we may write $f(x) = (a,b)$ and $f(y) = (a,c)$ where $a\in (2^\N)^k$, $b,c \in (2^\N)^{\N\setminus k}$, so that the triplet $(a,b,c)$ satisfies the conditions in Lemma~\ref{lemma: permutations} with respect to the comeager set $C_k\subset (2^\N)^Y$, where $Y = k \sqcup (\N\setminus k) \sqcup (\N\setminus k)$. It follows from Lemma~\ref{lemma: permutations} that there is some 
\begin{equation*}
(g,h_1,h_2) \in (\bigoplus_{n\in\N}\mathbb{Z}_2)^k \times (\bigoplus_{n\in\N}\mathbb{Z}_2)^{\N\setminus k}\times (\bigoplus_{n\in\N}\mathbb{Z}_2)^{\N\setminus k}    
\end{equation*}
so that
\begin{equation*}
    (g\cdot a, h_1\cdot b, h_2\cdot c)\in C_k,
\end{equation*} 
and so 
\begin{equation*}
    (g\cdot a, h_1\cdot b) \not\mathrel{E} (g\cdot a, h_2\cdot b).
\end{equation*}
Since $E$ extends $E_0^\N$, it is invariant under the action, and so 
\begin{equation*}
    f(x) = (a,b)\not\mathrel{E}(a,c) = f(y),
\end{equation*}
as required.
\end{proof}

\subsection{Complexity on comeager sets}\label{subsec: E3 retains complexity}
In the proof of Theorem~\ref{thm: E3 prime} we used the fact that for any comeager $C\subset (2^\N)^{\N}$, $E_0^\N \leq_B E_0^\N\restriction C$. We sketch a proof of this using the construction above.

Let $C\subset (2^\N)^\N$ be comeager. Similar to the above, define $D_{m}\subset (2^\N)^m$ as the intersection of all sets of the form 
\begin{equation*}
\set{\seq{\eta_0,\dots,\eta_{m-1}}\in (2^\N)^m}{C_{\gamma\cdot\seq{\eta_0,\dots,\eta_{m-1}}}\subset (2^\N)^{\N\setminus m} \textrm{ is comeager}},
\end{equation*}
for $\gamma \in (\bigoplus_{n\in\N}\mathbb{Z}_2)^m$. Define $f\colon E_0^\N \to_B E_0^\N$ as above, so that for any $x\in (2^\N)^\N$ and any $m\in\N$, $f(x)\restriction m \in D_m$. It follows from Lemma~\ref{lemma: permutations} that $\set{g\in (\bigoplus_{n\in\N}\mathbb{Z}_2)^\N}{g\cdot f(x)\in C}$ is comeager. It follows from \cite[Theorem~18.6]{Kechris-DST-1995} that there is a Borel map $h\colon (2^\N)^\N \to (\bigoplus_{n\in\N}\mathbb{Z}_2)^\N$ so that $h(x)\cdot f(x)\in C$ for all $x\in (2^\N)^\N$. We conclude that $x\mapsto h(x)\cdot f(x)$ is a Borel reduction of $E_0^\N$ to $E_0^\N\restriction C$.

\bibliographystyle{alpha}
\bibliography{bibliography}

\end{document}